\documentclass[11pt]{article}
\usepackage{amsmath}
\usepackage{fancyvrb}
\usepackage[all]{xy}
\usepackage{tabularx}
\usepackage{amscd}
\usepackage{mhchem}
\usepackage{amsthm}
\usepackage{amsfonts}

\def\Hom{\mathop{\rm Hom}\nolimits}
\def\Ext{\mathop{\rm Ext}\nolimits}

\def\Fac{\mathop{\rm Fac}\nolimits}

\def\id{\mathop{\rm id}\nolimits}
\def\mod{\mathop{\rm mod}\nolimits}

\def\add{\mathop{\rm add}\nolimits}

\def\End{\mathop{\rm End}\nolimits}
\def\tilt{\mathop{\rm tilt}\nolimits}

\usepackage{ulem}
\begin{document}

\newcommand{\nc}{\newcommand}
\def\PP#1#2#3{{\mathrm{Pres}}^{#1}_{#2}{#3}\setcounter{equation}{0}}
\def\ns{$n$-star}\setcounter{equation}{0}
\def\nt{$n$-tilting}\setcounter{equation}{0}
\def\Ht#1#2#3{{{\mathrm{Hom}}_{#1}({#2},{#3})}\setcounter{equation}{0}}
\def\qp#1{{${(#1)}$-quasi-projective}\setcounter{equation}{0}}
\def\mr#1{{{\mathrm{#1}}}\setcounter{equation}{0}}
\def\mc#1{{{\mathcal{#1}}}\setcounter{equation}{0}}
\def\HD{\mr{Hom}_{\mc{D}}}
\def\HC{\mr{Hom}_{\mc{C}}}
\def\AdT{\mr{Add}_{\mc{T}}}
\def\adT{\mr{add}_{\mc{T}}}
\def\Kb{\mc{K}^b(\mr{Proj}R)}
\def\kb{\mc{K}^b(\mc{P}_R)}
\def\AdpC{\mr{Adp}_{\mc{C}}}


\newtheorem{theorem}{Theorem}[section]
\newtheorem{proposition}[theorem]{Proposition}
\newtheorem{lemma}[theorem]{Lemma}
\newtheorem{corollary}[theorem]{Corollary}
\newtheorem{conjecture}[theorem]{Conjecture}
\newtheorem{question}[theorem]{Question}
\newtheorem{definition}[theorem]{Definition}
\newtheorem{example}[theorem]{Example}

\newtheorem{remark}[theorem]{Remark}
\def\Pf#1{{\noindent\bf Proof}.\setcounter{equation}{0}}
\def\>#1{{ $\Rightarrow$ }\setcounter{equation}{0}}
\def\<>#1{{ $\Leftrightarrow$ }\setcounter{equation}{0}}
\def\bskip#1{{ \vskip 20pt }\setcounter{equation}{0}}
\def\sskip#1{{ \vskip 5pt }\setcounter{equation}{0}}
\def\bg#1{\begin{#1}\setcounter{equation}{0}}
\def\ed#1{\end{#1}\setcounter{equation}{0}}
\def\KET{T^{^F\bot}\setcounter{equation}{0}}
\def\KEC{C^{\bot}\setcounter{equation}{0}}

\renewcommand{\thefootnote}{\fnsymbol{footnote}}
\setcounter{footnote}{0}
%
%


\title{\bf  Support $\tau$-Tilting Modules under Split-by-Nilpotent Extensions
\thanks{This work was partially supported by NSFC (No. 11571164) and a Project Funded
by the Priority Academic Program Development of Jiangsu Higher Education Institutions. } }
\footnotetext{
E-mail:~hpgao07@163.com,~huangzy@nju.edu.cn}
\smallskip
\author{\small Hanpeng Gao, Zhaoyong Huang\thanks{Corresponding author.}\\
{\it \footnotesize Department of Mathematics,  Nanjing University, Nanjing 210093, Jiangsu Province, P.R. China}}
\date{}
\maketitle
\baselineskip 15pt
%
%
\begin{abstract}
Let $\Gamma$ be a split extension of a finite-dimensional algebra $\Lambda$ by a nilpotent bimodule $_\Lambda E_\Lambda$,
and let $(T,P)$ be a pair in $\mod\Lambda$ with $P$ projective. We prove that $(T\otimes_\Lambda \Gamma_\Gamma, P\otimes_\Lambda \Gamma_\Gamma)$
is a support $\tau$-tilting pair in $\mod \Gamma$ if and only if $(T,P)$ is a support $\tau$-tilting pair in $\mod \Lambda$ and
$\Hom_\Lambda(T\otimes_\Lambda E,\tau T_\Lambda)=0=\Hom_\Lambda(P,T\otimes_\Lambda E)$. As applications, we obtain a
necessary and sufficient condition such that $(T\otimes_\Lambda \Gamma_\Gamma, P\otimes_\Lambda \Gamma_\Gamma)$
is support $\tau$-tilting pair for a cluster-tilted algebra $\Gamma$ corresponding to a tilted algebra $\Lambda$;
and we also get that if $T_1,T_2\in\mod\Lambda$ such that $T_1\otimes_\Lambda \Gamma$ and $T_2\otimes_\Lambda \Gamma$
are support $\tau$-tilting $\Gamma$-modules, then $T_1\otimes_\Lambda \Gamma$ is a left mutation of $T_2\otimes_\Lambda \Gamma$
if and only if $T_1$ is a left mutation of $T_2$.
\vspace{10pt}




\end{abstract}
%
\vskip 30pt

\section{Introduction}

In this paper, all algebras are finite-dimensional basic algebras over an algebraically closed field $k$.
For an algebra $\Lambda$, $\mod\Lambda$ is the category of finitely generated right $\Lambda$-modules
and $\tau$ is the Auslander-Reiten translation. We write $D:=\Hom_k(-,k)$

Mutation is an operation for a certain class of objects in a fixed category to construct a new object from a
given one by replacing a summand, which is possible only when the given object has two complements.
It is well known that tilting modules are fundamental in tilting theory.
Happel and Unger \cite{H} gave some necessary and sufficient conditions under which
mutation of tilting modules is possible; however, mutation of tilting modules is not always possible.
As a generalization of tilting modules, Adachi, Iyama and Reiten \cite{AIR} introduced support $\tau$-tilting modules
and showed that any almost complete support $\tau$-tilting module has exactly two complements. So, in this case,
mutation is always possible. Moreover, for a 2-Calabi-Yau triangulated category $\mathcal{C}$, it was showed in
\cite{AIR} that there is a close relation between cluster-tilting objects in $\mathcal{C}$ and
support $\tau$-tilting  $\Lambda$-modules, where $\Lambda$ is a 2-Calabi-Yau tilted algebra associated with $\mathcal{C}$.
Then Liu and Xie \cite{LX} proved that a maximal rigid object $T$ in $\mathcal{C}$ corresponds to a support $\tau$-tilting
$\End_{\mathcal{C}}(T)$-module.

Given two algebras $\Lambda$ and $\Gamma$, it is interesting to construct
a (support $\tau$-)tilting $\Gamma$-module from a (support $\tau$-)tilting $\Lambda$-module.
In \cite{AHS}, Assem, Happel and Trepode studied how to extend and restrict tilting modules for one-point extension algebras
by a projective module. Suarez \cite{SP} generalized this result to the case for
support $\tau$-tilting modules. More precisely, let $\Gamma=\Lambda[P]$ be the one-point extension of an algebra $\Lambda$
by a projective $\Lambda$-module $P$ and $e$ the identity of $\Lambda$. If $M_\Lambda$ is a basic support
$\tau$-tilting $\Lambda$-module, then $\Hom_\Gamma(\Gamma e,M_\Lambda)\oplus S$ is a basic support $\tau$-tilting $\Gamma$-module,
where $S$ is the simple module corresponding to the new point; conversely, if $T_\Gamma$ is a basic
support $\tau$-tilting $\Gamma$-module, then $\Hom_\Gamma(e\Gamma,T_\Gamma)$ is a basic support $\tau$-tilting $\Lambda$-module
\cite[Theorem A]{SP}.

Let $\Gamma$ be a split extension of an algebra $\Lambda$ by a nilpotent bimodule $_{\Lambda}E_{\Lambda}$,
that is, there exists a split surjective algebra
morphism $\Gamma\to\Lambda$ whose kernel $E$ is contained in the radical of $\Gamma$ \cite{AN,AZ}.
In particular, all relation extensions \cite{ABR,Z} and one-point extensions are split ones.
There are two functors $-\otimes _\Lambda \Gamma: \mod \Lambda\rightarrow \mod \Gamma$ and
$-\otimes _\Gamma \Lambda: \mod \Gamma\rightarrow \mod \Lambda$. Assem and Marmaridis \cite{AN} investigated the relationship
between (partial) tilting $\Gamma$-modules and (partial) tilting $\Lambda$-modules by using these two functors.
Analogously, we will investigate the relationship between support $\tau$-tilting $\Gamma$-modules and support
$\tau$-tilting $\Lambda$-modules. This paper is organized as follows.

In Section 2, we give some terminology and some preliminary results.

In Section 3, we first prove the following

\begin{theorem}\label{1.1} {\rm (Theorem \ref{3.1})}
Let $\Gamma$ be a split extension of $\Lambda$ by a nilpotent bimodule $_\Lambda E_\Lambda$. If
$(T,P)$ is a pair in $\mod \Lambda$ with $P$ projective, then the following statements are equivalent.
\begin{enumerate}
\item[(1)] $(T\otimes_\Lambda \Gamma_\Gamma,P\otimes_\Lambda \Gamma_\Gamma)$ is a support $\tau$-tilting pair in $\mod \Gamma$.
\item[(2)] $(T,P)$ is a support $\tau$-tilting pair in $\mod \Lambda$ and
$$\Hom_\Lambda(T\otimes_\Lambda E,\tau T_\Lambda)=0=\Hom_\Lambda(P,T\otimes_\Lambda E).$$
\end{enumerate}
\end{theorem}

As a consequence, we get that if $\Gamma$ is a cluster-tilted algebra corresponding to a tilted algebra $\Lambda$
and $(T,P)$ is a pair in $\mod \Lambda$ with $P$ projective, then $(T\otimes_\Lambda \Gamma_\Gamma, P\otimes_\Lambda \Gamma_\Gamma)$
is a support $\tau$-tilting pair in $\mod \Gamma$ if and only if $(T,P)$ is a support $\tau$-tilting pair in $\mod \Lambda$ and
$\Hom_\Lambda(\tau^{-1}\Omega^{-1}T_\Lambda,\tau T_\Lambda)=0=\Hom_\Lambda(P,\tau^{-1}\Omega^{-1}T_\Lambda)$ (Proposition \ref{3.4}).

Moreover, we have the following

\begin{theorem}\label{1.2} {\rm (Theorem \ref{3.10})}
Let $\Gamma$ be a split extension of $\Lambda$ by a nilpotent bimodule $_\Lambda E_\Lambda$.
Let $T_1,T_2\in\mod\Lambda$ such that $T_1\otimes_\Lambda \Gamma$ and $T_2\otimes_\Lambda \Gamma$ are support $\tau$-tilting $\Gamma$-modules.
Then the following statements are equivalent.
\begin{enumerate}
\item[(1)] $T_1\otimes_\Lambda \Gamma$ is a left mutation of $T_2\otimes_\Lambda \Gamma$.
\item[(2)] $T_1$ is a left mutation of $T_2$.
\end{enumerate}
\end{theorem}

The Hasse (exchange) quiver $Q(s\tau$-$\tilt \Lambda)$ of $\Lambda$ consists of the set of vertices which are support $\tau$-tilting $\Lambda$-modules $T$
and those  arrows from $T$ to its left mutation. So Theorem \ref{1.2} shows that if $T_1,T_2\in\mod\Lambda$ such that $T_1\otimes_\Lambda \Gamma$
and $T_2\otimes_\Lambda \Gamma$ are support $\tau$-tilting $\Gamma$-modules, then there exists an arrow from $T_1\otimes_\Lambda \Gamma$
to $T_2\otimes_\Lambda \Gamma$ in $Q(s\tau$-$\tilt \Gamma)$ if and only if there exists an arrow from $T_1$ to $T_2$ in $Q(s\tau$-$\tilt \Lambda)$.

In Section 4, we give two examples to illustrate our results.

\section{Preliminaries}


Let $\Lambda$ be an algebra. For a module $M\in\mod \Lambda$, $|M|$ is the number of pairwise non-isomorphic direct summands of $M$,
$\add M$ is the full subcategory of $\mod \Lambda$ consisting of modules isomorphic to direct summands
of finite direct sums of copies of $M$, and $\Fac M$ is the full subcategory of $\mod \Lambda$ consisting
of modules isomorphic to factor modules of finite direct sums of copies of $M$. The injective dimension
and the first cosyzygy of $M$ are denoted by $\id_{\Lambda}M$ and $\Omega^{-1}M$ respectively.

\subsection{$\tau$-tilting theory}

\begin{definition}\label{2.1} {\rm (\cite[Definition 0.1]{AIR})
A module $M\in\mod\Lambda$ is called
\begin{enumerate}
\item[(1)] {\it $\tau$-rigid} if $\Hom_\Lambda(M,\tau M)=0$;
\item[(2)] {\it $\tau$-tilting} (respectively, {\it almost complete $\tau$-tilting}) if it is $\tau$-rigid and
$|M|=|\Lambda|$ (respectively, $|M|=|\Lambda|-1$);
\item[(3)] {\it support $\tau$-tilting} if it is a $\tau$-tilting $\Lambda/<e>$-module for some idempotent $e$ of $\Lambda$.
\end{enumerate}}
\end{definition}

The next result shows a $\tau$-rigid module may be extended to a $\tau$-tilting module.

\begin{theorem}\label{2.2} {\rm (\cite[Theorem 2.10]{AIR})}
Any basic $\tau$-rigid $\Lambda$-module is a direct summand of a $\tau$-tilting $\Lambda$-module.
\end{theorem}

\begin{lemma}\label{2.3} {\rm (\cite[Proposition 2.4]{AIR})}
Let $X\in\mod\Lambda$ and
$$P_1 \stackrel{f_0}{\longrightarrow}P_0 {\longrightarrow} X{\longrightarrow} 0$$ be a projective presentation
of $X$ in $\mod\Lambda$. For any $Y\in\mod \Lambda$, if $\Hom_\Lambda(f_0,Y)$ is epic, then $\Hom_\Lambda(Y,\tau X)=0$.
Moreover, the converse holds if the projective presentation is minimal.
\end{lemma}

Sometimes, it is convenient to view support $\tau$-tilting modules and $\tau$-rigid modules as certain pairs of modules in $\mod \Lambda$.

\begin{definition} \label{2.4} {\rm (\cite[Definition 0.3]{AIR})
Let $(M,P)$ be a pair in $\mod\Lambda$ with $P$ projective.
\begin{enumerate}
\item[(1)] The pair $(M, P)$ is called a {\it $\tau$-rigid pair} if $M$ is $\tau$-rigid and $\Hom_\Lambda(P,M)=0$.
\item[(2)] The pair $(M, P)$ is called a {\it support $\tau$-tilting pair} (respectively, {\it almost complete $\tau$-tilting pair})
if it is $\tau$-rigid and $|M|+|P|=|\Lambda|$ (respectively, $|M|+|P|=|\Lambda|-1$).
\end{enumerate}}
\end{definition}

Note that $(M,P)$ is a support $\tau$-tilting pair if and only if $M$ is a $\tau$-tilting $\Lambda/<e>$-module,
where $e\Lambda\cong P$. Hence, $M$ is a $\tau$-tilting $\Lambda$-module if and only if $(M,0)$ is a support $\tau$-tilting pair.

Let $(U,Q)$ be an almost complete $\tau$-tilting pair and $X\in\mod\Lambda$ indecomposable. We say that
$(X,0)$ (respectively, $(0,X)$) is a {\it complement} of $(U,Q)$ if $(U\oplus X,Q)$ (respectively, $(U,Q\oplus X)$)
is support $\tau$-tilting. It follows from \cite[Theorem 2.18]{AIR} that
any basic almost complete $\tau$-tilting pair in $\mod \Lambda$ has exactly two complements.
Two support $\tau$-tilting pairs $(T,P)$ and $(\widetilde{T},\widetilde{P})$ in $\mod\Lambda$
are called {\it mutations}  of each other if they have the same direct summand $(U,Q)$
which is an almost complete $\tau$-tilting pair. In this case, we write
$(\widetilde{T}, \widetilde{P})=\mu_{X}(T, P)$ (simply $\widetilde{T}=$$\mu_{X}T$)
if the indecomposable module $X$ satisfies either $T=U \oplus X$ or $P=Q\oplus X$.

\begin{definition}\label{2.5} {\rm (\cite[Definition 2.28]{AIR})
Let $T=U\oplus X$ and $\widetilde{T}$ be two support $\tau$-tilting $\Lambda$-modules such that
$\widetilde{T}=\mu_{X}T$ with $X$ indecomposable. Then $\widetilde{T}$ is called a {\it left mutation}
(respectively, {\it right mutation}) of $T$, denoted by $\widetilde{T}=\mu^-_{ X}T$
(respectively, $\widetilde{T}=\mu^+_{ X}T$), if $X\notin \Fac U$ (respectively, $X\in \Fac U$).}
\end{definition}

\begin{definition}\label{2.6} {\rm (\cite[Definition 2.29]{AIR})
The {\it support $\tau$-tilting quiver} $Q(s\tau$-$\tilt \Lambda)$ of $\Lambda$ is defined as follows.
\begin{enumerate}
\item[(1)] The set of vertices consists of the isomorphisms classes of basic support $\tau$-tilting $\Lambda$-modules.
\item[(2)] We draw an  arrow from $T$ to its left mutation.
\end{enumerate}}
\end{definition}

\subsection{Split-by-nilpotent extensions}

Let $\Lambda$ and $\Gamma$ be two algebras.

\begin{definition}\label{2.7} {\rm (\cite[Definition 1.1]{AZ})
We say that $\Gamma$ is a {\it split extension of $\Lambda$ by the nilpotent bimodule $_\Lambda E_\Lambda$}, or simply
a {\it split-by-nilpotent extension} if there exists a split surjective algebra
morphism $\Gamma\to\Lambda$ whose kernel $E$ is contained in the radical of $\Gamma$.}
\end{definition}

Let $\Gamma$ be a split-by-nilpotent extension of $\Lambda$ by the nilpotent bimodule $_\Lambda E_\Lambda$.
Clearly, the short exact sequence of $\Lambda$-$\Lambda$-bimodules
$$0 {\longrightarrow}_\Lambda E_\Lambda {\longrightarrow}_\Lambda \Gamma_\Lambda{\longrightarrow}\Lambda{\longrightarrow}0$$
splits. Therefore, there exists an isomorphism $_\Lambda \Gamma_\Lambda\cong \Lambda\oplus _\Lambda E_\Lambda$.
The module categories  over $\Lambda$ and $\Gamma$ are related by the following functors
\begin{center}
$-\otimes _\Lambda \Gamma: \mod \Lambda\rightarrow \mod \Gamma$,  ~~~~  $-\otimes _\Gamma \Lambda: \mod \Gamma \rightarrow \mod \Lambda$,
\end{center}
\begin{center}
$\Hom_\Lambda(\Gamma_\Lambda, -): \mod \Lambda\rightarrow \mod \Gamma$,  ~~~~ $\Hom_\Gamma(\Lambda_\Gamma, -): \mod \Gamma\rightarrow \mod \Lambda$.
\end{center}
Moreover, we have $$-\otimes _\Lambda \Gamma_\Gamma\otimes _\Gamma \Lambda\cong 1_{\mod \Lambda},$$
$$\Hom_\Gamma(\Lambda_\Gamma, \Hom_\Lambda(\Gamma_\Lambda, -))\cong 1_{\mod \Lambda}.$$

\begin{lemma}\label{2.8}
Let $\Gamma$ be a split-by-nilpotent extension of $\Lambda$. Then for any $M\in\mod\Lambda$, we have
\begin{enumerate}
\item[(1)] There exists a bijective correspondence between the isomorphism classes of indecomposable summands of
$M$ in $\mod \Lambda$ and the isomorphism classes of indecomposable summands of $M_\Lambda\otimes_\Lambda \Gamma$ in $\mod \Gamma$,
given by $N_\Lambda\rightarrow N_\Lambda\otimes_\Lambda \Gamma$.
\item[(2)] $|M_\Lambda|=|M_\Lambda\otimes_\Lambda \Gamma|$.
\item[(3)] Any indecomposable projective module in $\mod \Gamma$ is the form $P\otimes_\Lambda \Gamma$,
where $P$ is indecomposable projective in $\mod \Lambda$. In particular, $|\Lambda|=|\Gamma|$.
\end{enumerate}
\end{lemma}

\begin{proof}
The assertion (1) is \cite[Lemma 1.2]{AN}. The latter two assertions follow immediately from (1).
\end{proof}

\begin{lemma}\label{2.9} {\rm (\cite[Lemma 2.1]{AN})}
Let $\Gamma$ be a split-by-nilpotent extension of $\Lambda$. Then for any $M\in\mod\Lambda$, there exists the following isomorphism
\begin{center}
$\tau(M\otimes_\Lambda \Gamma)\cong \Hom_\Lambda(_\Gamma \Gamma_\Lambda, \tau M_\Lambda)$.
\end{center}
\end{lemma}

\section{Main results}

In this section, assume that $\Gamma$ is a split extension of $\Lambda$ by the nilpotent bimodule $_\Lambda E_\Lambda$.

\subsection{$\tau$-tilting and $\tau$-rigid modules}

The following result is a $\tau$-version of \cite[Theorem A]{AN}.

\begin{theorem}\label{3.1}
Let $(T,P)$ be a pair in $\mod \Lambda$ with $P$ projective. Then the following statements are equivalent.
\begin{enumerate}
\item[(1)] $(T\otimes_\Lambda \Gamma_\Gamma,P\otimes_\Lambda \Gamma_\Gamma)$ is a support $\tau$-tilting pair in $\mod \Gamma$.
\item[(2)] $(T,P)$ is a support $\tau$-tilting pair in $\mod \Lambda$ and
$$\Hom_\Lambda(T\otimes_\Lambda E,\tau T_\Lambda)=0=\Hom_\Lambda(P,T\otimes_\Lambda E).$$
\end{enumerate}
\end{theorem}

\begin{proof}
By Lemma \ref{2.8}(2), we have $|T|+|P|=|T\otimes_\Lambda \Gamma|+|P\otimes_\Lambda \Gamma|$.
Hence, $|T|+|P|=|\Lambda|$ if and only if $|T\otimes_\Lambda \Gamma|+|P\otimes_\Lambda \Gamma|=|\Gamma|$
Lemma \ref{2.8}(3).

Let $T,P\in\mod \Lambda$. Then there are the following two isomorphisms
\begin{equation*}
\begin{split}
\Hom_\Gamma(T\otimes_\Lambda \Gamma, \tau(T\otimes_\Lambda \Gamma))
&\cong \Hom_\Gamma(T\otimes_\Lambda \Gamma, \Hom_\Lambda(_\Gamma \Gamma_\Lambda, \tau T_\Lambda))\ (\makebox{by Lemma \ref{2.9}})\\
&\cong  \Hom_\Lambda(T\otimes_\Lambda \Gamma\otimes _\Gamma \Gamma_\Lambda, \tau T_\Lambda)\ (\makebox{by the adjunction isomorphism})\\
&\cong  \Hom_\Lambda(T\otimes_\Lambda \Gamma_\Lambda, \tau T_\Lambda)\\
&\cong  \Hom_\Lambda(T\otimes_\Lambda(\Lambda\oplus E)_\Lambda, \tau T_\Lambda)\\
&\cong  \Hom_\Lambda(T, \tau T_\Lambda)\oplus \Hom_\Lambda(T\otimes_\Lambda E,\tau T_\Lambda),
\end{split}
\end{equation*}
and
\begin{equation*}
\begin{split}
\Hom_\Gamma(P\otimes_\Lambda \Gamma, T\otimes_\Lambda \Gamma)
&\cong \Hom_\Lambda(P_\Lambda, \Hom_\Gamma(_\Lambda \Gamma_\Gamma, T\otimes_\Lambda \Gamma))\ (\makebox{by the adjunction isomorphism})\\
&\cong  \Hom_\Lambda(P_\Lambda, T\otimes_\Lambda \Gamma_\Lambda)\\
&\cong  \Hom_\Lambda(P_\Lambda, T\otimes_\Lambda(\Lambda\oplus E)_\Lambda)\\
&\cong  \Hom_\Lambda(P_\Lambda,  T_\Lambda)\oplus \Hom_\Lambda(P_\Lambda, T\otimes_\Lambda E).
\end{split}
\end{equation*}
\end{proof}

Note that $T$ is a $\tau$-tilting $\Lambda$-module if and only if $(T,0)$ is a support $\tau$-tilting pair in $\mod \Lambda$.
The following corollary is an immediate consequence of Theorem \ref{3.1}.

\begin{corollary}\label{3.2}
For a module $T\in\mod\Lambda$, the following statements are equivalent.
\begin{enumerate}
\item[(1)] $T\otimes_\Lambda \Gamma_\Gamma$ is $\tau$-tilting in $\mod \Gamma$.
\item[(2)] $T$ is $\tau$-tilting in $\mod \Lambda$ and $\Hom_\Lambda(T\otimes_\Lambda E,\tau T_\Lambda)=0$.
\end{enumerate}
\end{corollary}

Let $T\in\mod\Lambda$ be $\tau$-rigid. Assume that $E_\Lambda$ is generated by $T$,
that is, there exists an epimorphism
$$T^{(n)}{\longrightarrow}E_\Lambda{\longrightarrow}0$$ in $\mod\Lambda$ for some $n\geq 1$.
Applying the functor $\Hom_\Lambda(-,\tau T_\Lambda)$ to it yields a monomorphism
$$0{\longrightarrow}\Hom_\Lambda(E_\Lambda, \tau T_\Lambda){\longrightarrow}\Hom_\Lambda(T^{(n)},\tau T_\Lambda)=0.$$
So $\Hom_\Lambda(E_\Lambda, \tau T_\Lambda)=0$, and hence
$$\Hom_\Lambda(T\otimes_\Lambda E,\tau T_\Lambda)\cong \Hom_\Lambda(T_\Lambda,\Hom_\Lambda(_\Lambda E_\Lambda,\tau T_\Lambda))=0.$$
Thus by Theorem \ref{3.1} and Corollary \ref{3.2}, we have the following result.

\begin{corollary}\label{3.3}
Let $(T,P)$ be a pair in $\mod \Lambda$ with $P$ projective. If $E_{\Lambda}$ is generated by $T$,
then the following statements are equivalent.
\begin{enumerate}
\item[(1)] $(T\otimes_\Lambda \Gamma _R,P\otimes_\Lambda \Gamma_\Gamma)$ is a support $\tau$-tilting pair in $\mod \Gamma$.
\item[(2)] $(T,P)$ is a support $\tau$-tilting pair in $\mod \Lambda$ and $\Hom_\Lambda(P,T\otimes_\Lambda E)=0$.
\end{enumerate}
Moreover, $T\otimes_\Lambda \Gamma_\Gamma$ is $\tau$-tilting in $\mod \Gamma$  if and only if $T$ is  $\tau$-tilting in $\mod \Lambda$.
\end{corollary}

Let $A$ be a hereditary algebra and $\mathcal{D}^{b}(\mod A)$ the bounded derived category of
$\mod A$. The {\it cluster category} $\mathcal{C}_A$ is defined by the orbit category of
$\mathcal{D}^{b}(\mod A)$ under the action of the functor $\tau^{-1}[1]$, where $[1]$ is the shift functor; and
a {\it tilting object} $\widetilde{T}$ in $\mathcal{C}_A$ is an object such that
$\Ext^1_{\mathcal{C}_A}(\widetilde{T},\widetilde{T})=0$ and $|\widetilde{T}|=|A|$ (\cite{B}).
The endomorphism algebra of $\widetilde{T}$ is called {\it cluster-tilted} (\cite{BM}).
It was shown in \cite[Theorem 3.4]{ABR} that, if $\Lambda$ is a tilted algebra, then the relation extension
of $\Lambda$ by $\Ext^2_\Lambda(D\Lambda, \Lambda)$ is cluster-tilted. Moreover, all cluster-tilted algebras
are of this form. In this case, we say $\Gamma$  is a cluster-tilted algebra corresponding to the tilted algebra $\Lambda$.

\begin{proposition} \label{3.4}
Let $\Gamma$ be a cluster-tilted algebra corresponding to the tilted algebra $\Lambda$
and $(T,P)$ a pair in $\mod \Lambda$ with $P$ projective. Then the following statements are equivalent.
\begin{enumerate}
\item[(1)] $(T\otimes_\Lambda \Gamma_\Gamma, P\otimes_\Lambda \Gamma_\Gamma)$ is a support $\tau$-tilting pair in $\mod \Gamma$.
\item[(2)] $(T,P)$ is a support $\tau$-tilting pair in $\mod \Lambda$ and
$$\Hom_\Lambda(\tau^{-1}\Omega^{-1}T_\Lambda,\tau T_\Lambda)=0=\Hom_\Lambda(P,\tau^{-1}\Omega^{-1}T_\Lambda).$$
\end{enumerate}
\end{proposition}

\begin{proof}
Since the global dimension of the tilted algebra $\Lambda$ is at most $2$, we have
$$T\otimes_\Lambda \Ext^2_\Lambda(D\Lambda, \Lambda)\cong \tau^{-1}\Omega^{-1}T$$
by \cite[Proposition 4.1]{SR}.  Now the assertion follows from Theorem \ref{3.1}.
\end{proof}

If $\id_{\Lambda}T\leq 1$, then $\tau^{-1}\Omega^{-1}T=0$. So by Proposition \ref{3.4}, we have the following corollary.

\begin{corollary}\label{3.5}
Let $\Gamma$ be a cluster-tilted algebra corresponding to the tilted algebra $\Lambda$
and $(T,P)$ a pair in $\mod \Lambda$ with $\id_{\Lambda}T\leq 1$ and $P$ projective. Then the following statements are equivalent.
\begin{enumerate}
\item[(1)] $(T\otimes_\Lambda \Gamma_\Gamma, P\otimes_\Lambda \Gamma_\Gamma)$ is a support $\tau$-tilting pair in $\mod \Gamma$.
\item[(2)] $(T,P)$ is a support $\tau$-tilting pair in $\mod \Lambda$.
\end{enumerate}
In particular, $T\otimes_\Lambda \Gamma_\Gamma$ is a $\tau$-tilting $\Gamma$-module if and only if $T$ is a $\tau$-tilting $\Lambda$-module.
\end{corollary}

Let $\mathcal{C}$ be a full subcategory of $\mod \Lambda$. We write
$$\mathcal{C}^{\perp}:=\{M\in \mod \Lambda\mid\Hom_\Lambda(C,M)=0\ \text{for any}\ C\in\mathcal{C}\},$$
$$^{\perp}\mathcal{C}:=\{M\in \mod \Lambda\mid\Hom_\Lambda(M,C)=0\ \text{for any}\ C\in\mathcal{C}\}.$$
Recall that a pair $(\mathcal{T}, \mathcal{F})$ of full subcategories of $\mod \Lambda$ is called a {\it torsion pair}
if $\mathcal{T}^\perp=\mathcal{F}$ and $^\perp \mathcal{F}=\mathcal{T}$. Given a $\tau$-tilting $\Lambda$-module $T$, it will induce a torsion pair
$(\mathcal{T}(T), \mathcal{F}(T)):=(^\perp(\tau T), T^{\perp})$ (\cite{AIR}).

\begin{proposition}\label{3.6}
Let $X_\Gamma\in\mod\Gamma$, and let $T\in\mod\Lambda$ be $\tau$-tilting such that $\Hom_\Lambda(T\otimes_\Lambda E,\tau T_\Lambda)=0$.
Then we have
\begin{enumerate}
\item[(1)] $X_\Gamma\in$$ \mathcal{T}(T\otimes_\Lambda \Gamma)$ if and only if $X\otimes _\Gamma \Gamma_\Lambda\in\mathcal{T}(T)$.
\item[(2)] $X_{\Gamma}\in \mathcal{F}(T\otimes_\Lambda \Gamma)$ if and only if $\Hom_\Gamma(_\Lambda \Gamma_\Gamma, X_\Gamma)\in \mathcal{F}(T)$.
\end{enumerate}
\end{proposition}

\begin{proof}
Since $\Hom_\Lambda(T\otimes_\Lambda E,\tau T_\Lambda)=0$, we have that $T\otimes_\Lambda \Gamma$ is a $\tau$-tilting $\Gamma$-module
by Corollary \ref{3.2} and it will induce a torsion pair. Note that there exist two isomorphisms
$$\Hom_\Gamma(X_\Gamma, \tau(T\otimes_\Lambda \Gamma))\cong \Hom_\Gamma(X_\Gamma, \Hom_\Lambda(_\Gamma \Gamma_\Lambda, \tau T_\Lambda))
\cong\Hom_\Lambda(X\otimes _\Gamma \Gamma_\Lambda, \tau T_\Lambda),$$
$$\Hom_\Gamma(T\otimes_\Lambda \Gamma, X_\Gamma)\cong \Hom_\Lambda(T_\Lambda, \Hom_\Gamma(_\Lambda \Gamma_\Gamma, X_\Gamma)).$$
The result is obvious.
\end{proof}

For a $\Gamma$-module $U_\Gamma$, $U\otimes_\Gamma \Lambda$ is a $\Lambda$-module. If $U_\Gamma$ is $\tau$-tilting and
$U\otimes_\Gamma \Lambda\otimes_\Lambda \Gamma \cong U_\Gamma$, then $U\otimes_\Gamma \Lambda$ is a $\tau$-tilting $\Lambda$-module by Theorem \ref{3.1}.
As a slight generalization of this observation, the following result gives a converse construction of Corollary \ref{3.2}.

\begin{proposition}\label{3.7}
Assume that $U_\Gamma$ is a $\Gamma$-module such that $U\otimes_\Gamma \Lambda\otimes_\Lambda \Gamma\in\add U_\Gamma$.
\begin{enumerate}
\item[(1)] If $U_\Gamma$ is $\tau$-rigid, then $U\otimes_\Gamma \Lambda$ is a $\tau$-rigid $\Lambda$-module.
\item[(2)] If $U_\Gamma$ is $\tau$-tilting and $U\otimes_\Gamma \Lambda$ is basic, then $U\otimes_\Gamma \Lambda$ is a $\tau$-tilting $\Lambda$-module.
\end{enumerate}
\end{proposition}

\begin{proof}
(1) Let $U_\Gamma$ be $\tau$-rigid and
$$P_1\otimes_\Lambda \Gamma \stackrel{f_0}{\longrightarrow}P_0\otimes_\Lambda \Gamma {\longrightarrow} U_\Gamma{\longrightarrow} 0$$
a minimal projective presentation of $U$ in $\mod \Gamma$ with $P_0,P_1$ projective $\Lambda$-modules.
Applying the functor $-\otimes_\Gamma \Lambda$ to it, we obtain a projective presentation
$$P_1 \stackrel{f_0\otimes 1_\Lambda}{\longrightarrow}P_0 {\longrightarrow} U_\Gamma\otimes_\Gamma \Lambda{\longrightarrow} 0$$
of $U\otimes_\Gamma \Lambda$ in $\mod\Lambda$.
To prove that $U\otimes_\Gamma \Lambda$ is a $\tau$-rigid $\Lambda$-module, it suffices to show
$\Hom_\Lambda(f_0\otimes 1_\Lambda,U\otimes_\Gamma \Lambda)$ is epic by Lemma \ref{2.3}.

Let $g\in \Hom_\Lambda(P_1,U\otimes_\Gamma \Lambda)$. Then $g\otimes 1_\Gamma\in
\Hom_\Gamma(P_1\otimes _\Lambda \Gamma,U\otimes_\Gamma \Lambda\otimes_\Lambda \Gamma)$.
By assumption, we have $U\otimes_\Gamma \Lambda\otimes_\Lambda \Gamma $ $\in\add U_\Gamma$.
Without loss of generality, assume that $U\otimes_\Gamma \Lambda\otimes_\Lambda \Gamma $ is basic,
and hence it is a direct summand of $U_\Gamma$. Then there exist a canonical embedding
$\lambda: U\otimes_\Gamma \Lambda\otimes_\Lambda \Gamma \to U_\Gamma$
and a canonical epimorphism
$\pi:U_\Gamma \to U\otimes_\Gamma \Lambda\otimes_\Lambda \Gamma$ such that $\pi\lambda=1_{U\otimes_\Gamma \Lambda\otimes_\Lambda \Gamma}$.
Consider the following diagram
\[\xymatrix{P_1\otimes_\Lambda\Gamma\ar[r]^{f_0}\ar[d]_{g\otimes 1_\Gamma}&P_0\otimes_\Lambda \Gamma\ar[r]\ar@{.>}[ld]_{\pi i}\ar@{.>}[ldd]^{i}&U_\Gamma\ar[r]&0\\
U\otimes_\Gamma \Lambda\otimes_\Lambda \Gamma\ar@<.5ex>[d]^{\lambda}&&&\\
 U_\Gamma. \ar@<.5ex>[u]^{\pi} &&&
}\]
Since $\Hom_\Gamma(f_0,U_\Gamma)$ is epic by Lemma \ref{2.3},
there exists $i\in \Hom_\Gamma(P_0\otimes _\Lambda \Gamma, U_\Gamma)$ such that $\lambda(g\otimes 1_\Gamma)=if_0$. Then we have
$$g\otimes 1_\Gamma=1_{U\otimes_\Gamma \Lambda\otimes_\Lambda \Gamma}(g\otimes 1_\Gamma)=\pi\lambda (g\otimes 1_\Gamma)=(\pi i)f_0,$$
and
$$g\cong g\otimes 1_\Gamma\otimes 1_\Lambda \cong ((\pi i) f_{0})\otimes 1_\Lambda \cong ((\pi i)\otimes 1_\Lambda)(f_0\otimes 1_\Lambda).$$
Therefore $\Hom_\Lambda(f_{0}\otimes 1_\Lambda, U\otimes_\Gamma \Lambda)$ is epic.

(2) If $U_\Gamma$ is $\tau$-tilting, then $|U\otimes_\Gamma \Lambda|\geq |U_\Gamma|=|\Gamma|=|\Lambda|$ by Lemma \ref{2.8}(3).
Thus $U\otimes_\Gamma \Lambda$ is a $\tau$-tilting $\Lambda$-module when it is basic by (1) and Theorem \ref{2.2}.
\end{proof}

However, $U\otimes_\Gamma \Lambda$ may not be basic even if $U_\Gamma$ is basic. Let $M(U\otimes_\Gamma \Lambda)$
stand for the maximal basic direct summand of $U\otimes_\Gamma \Lambda$, that is, the direct sum of
all indecomposable direct summands of $U\otimes_\Gamma \Lambda$ which are pairwise non-isomorphic.

\begin{example}\label{3.8}
{\rm Let $\Lambda$ be the algebra given by the quiver
$$1\longrightarrow 2$$
and $\Gamma$ the algebra given by the quiver
\[\xymatrix{1\ar@/^/[r]^\alpha & \ar@/^/[l]^\beta 2
}\]
with the relation $\alpha\beta=0$. Then $\Gamma$ is the split extension of $\Lambda$ by the nilpotent $E$ generated by $\beta$
and $U_\Gamma=S_{2}\oplus e_2\Gamma$ is a $\tau$-tilting $\Gamma$-module, where $S_2$ is the simple $\Gamma$-module corresponding
to the vertex $2$. Applying the functor $-\otimes_\Gamma \Lambda$ to the projective presentation
$$0{\longrightarrow}e_1\Gamma {\longrightarrow} (e_2\Gamma)^2 {\longrightarrow} U_\Gamma {\longrightarrow} 0$$
of $U_\Gamma$, we get an exact sequence
$$e_1\Lambda \buildrel {0}\over {\longrightarrow} (e_2\Lambda)^2{\longrightarrow} U\otimes_\Gamma \Lambda{\longrightarrow} 0$$
in $\mod \Lambda$. So $U\otimes_\Gamma \Lambda\cong (e_2\Lambda)^2$ and it is not basic. Note that
$U\otimes_\Gamma \Lambda\otimes_\Lambda \Gamma $ $\in\add U_\Gamma$
because $U\otimes_\Gamma \Lambda\otimes_\Lambda \Gamma\cong (e_2\Gamma)^2$. Moreover,
$M(U\otimes_\Gamma \Lambda)\cong e_2\Lambda$ is a support $\tau$-tilting $\Lambda$-module.}
\end{example}

We do not know whether the answer to the following question is positive or not.

\begin{question}\label{3.9}
Under the condition of Proposition \ref{3.7}, if $U_\Gamma$ is $\tau$-tilting,
is then $M(U\otimes_\Gamma \Lambda)$ a support $\tau$-tilting $\Lambda$-module?
\end{question}

\subsection{Left mutations}

Let $T$ be a support $\tau$-tilting $\Lambda$-module such that $T\otimes_\Lambda \Gamma$ is a support $\tau$-tilting $\Gamma$-module.
By Lemma \ref{2.8}(1), all indecomposable summands of $T_\Lambda\otimes_\Lambda \Gamma$ are of the forms $X\otimes_\Lambda \Gamma$
for some indecomposable summand $X$ of $T$. In the following, we investigate the relationship between $Q(s\tau$-$\tilt\Lambda)$
and $Q(s\tau$-$\tilt\Gamma)$.

\begin{theorem}\label{3.10}
Let $T_1,T_2\in\mod\Lambda$ such that $T_1\otimes_\Lambda \Gamma$ and $T_2\otimes_\Lambda \Gamma$ are support $\tau$-tilting $\Gamma$-modules.
Then the following statements are equivalent.
\begin{enumerate}
\item[(1)] $T_1\otimes_\Lambda \Gamma$ is a left mutation of $T_2\otimes_\Lambda \Gamma$.
\item[(2)] $T_1$ is a left mutation of $T_2$.
\end{enumerate}
\end{theorem}

\begin{proof}
$(1)\Rightarrow (2)$ Since $T_1\otimes_\Lambda \Gamma$ and $T_2\otimes_\Lambda \Gamma$ are support $\tau$-tilting $\Gamma$-modules by assumption,
$T_1$ and $T_2$ are support $\tau$-tilting $\Lambda$-modules by Theorem \ref{3.1}.

Let $T_1\otimes_\Lambda \Gamma=\mu^-_{X\otimes_\Lambda \Gamma}(T_2\otimes_\Lambda \Gamma)$ for some indecomposable $\Lambda$-module $X$.
Assume that $(T_1\otimes_\Lambda \Gamma, P_1\otimes_\Lambda \Gamma)$ and $(T_2\otimes_\Lambda \Gamma, P_2\otimes_\Lambda \Gamma)$ are
support $\tau$-tilting pairs having the same almost complete support $\tau$-tilting pair $(U\otimes_\Lambda \Gamma, Q\otimes_\Lambda \Gamma)$,
where $U$ and $Q$ are $\Lambda$-modules. Then by Lemma \ref{2.8}(1), $(T_1, P_1)$ and $(T_2, P_2)$ have the same almost complete support $\tau$-tilting pair
$(U, Q)$ and are mutations of each other.

Because $T_2\otimes_\Lambda \Gamma=(X\otimes_\Lambda \Gamma) \oplus (U\otimes_\Lambda \Gamma)$.
we have $T_2\cong X\oplus U$. It suffices to show $X\notin \Fac U$. Otherwise, there exists an epimorphism
$U^{(n)}{\rightarrow}X{\rightarrow}0$ in $\mod\Lambda$ for some $n\geq 1$, which yields an epimorphism
$U^{(n)}\otimes_\Lambda \Gamma(\cong (U\otimes_\Lambda \Gamma)^{(n)}){\rightarrow}X\otimes_\Lambda \Gamma{\rightarrow}0$ in $\mod \Gamma$.
It implies $X\otimes_\Lambda \Gamma\in \Fac(U\otimes_\Lambda \Gamma)$, a contradiction.

Similarly, we get $(2)\Rightarrow (1)$.
\end{proof}

As an immediate consequence of Theorem \ref{3.10} and its proof, we get the following

\begin{corollary}\label{3.11}
Let $T_1,T_2\in\mod\Lambda$ such that $T_1\otimes_\Lambda \Gamma$ and $T_2\otimes_\Lambda \Gamma$
are support $\tau$-tilting $\Gamma$-modules, and let $X$ be the indecomposable $\Lambda$-module
as in the proof of Theorem \ref{3.10}. Then the following statements are equivalent.
\begin{enumerate}
\item[(1)] $T_1\otimes_\Lambda \Gamma=\mu^-_{X\otimes_\Lambda \Gamma}(T_2\otimes_\Lambda \Gamma)$.
\item[(2)] $T_1=\mu^-_{X}T_2$.
\end{enumerate}
\end{corollary}

Let $Q=(Q_0, Q_1)$ be a quiver. A subquiver $\widehat{Q}=(\widehat{Q}_0, \widehat{Q}_1)$ of $Q$ is called {\it full}
if $\widehat{Q}_1$ equals the set of all those arrows in $Q_1$ whose source and target both belong to $\widehat{Q}_0$
\cite[Chapter II]{ASS}. We use $fQ(s\tau$-$\tilt\Gamma)$ to denote the full subquiver of $Q(s\tau$-$\tilt\Gamma)$
whose  vertices are $T\otimes_\Lambda \Gamma$ where $T\in Q(s\tau$-$\tilt \Lambda)$, and use $fQ(s\tau$-$\tilt \Lambda)$ to denote
the full subquiver of $Q(s\tau$-$\tilt \Lambda)$ whose  vertices are those support $\tau$-tilting $\Lambda$-modules
$T$ such that $T\otimes_\Lambda \Gamma$ is a support $\tau$-tilting $\Gamma$-module. Corollary \ref{3.11} shows that
the underlying graph of  $fQ(s\tau$-$\tilt\Lambda)$ and $fQ(s\tau$-$\tilt\Gamma)$ coincide.
More precisely, if $T_1,T_2\in\mod\Lambda$ such that $T_1\otimes_\Lambda \Gamma$ and $T_2\otimes_\Lambda \Gamma$
are support $\tau$-tilting $\Gamma$-modules, then
there exists an arrow from $T_1\otimes_A\Gamma$ to $T_2\otimes_\Lambda \Gamma$ in $Q(s\tau$-$\tilt\Gamma)$ if and only if
there exists an arrow from $T_1$ to $T_2$ in  $Q(s\tau$-$\tilt\Lambda)$.

\subsection{A special case}

We now turn attention to one-point extensions. Let $\Lambda$ be an algebra and
$M\in \mod\Lambda$. The {\it one-point extension} of $\Lambda$ by $M$ is defined as the following matrix algebra
\begin{center}
$\Gamma=\left(\begin{matrix}
\Lambda & 0\\
M_\Lambda & k\\
\end{matrix}\right)$
\end{center}
with the ordinary matrix addition and the multiplication, and we write $\Gamma:=\Lambda[M]$ with $a$ the extension point.
Let $\Delta:=\Lambda \times k$, and let $E$ be the $(\Delta,\Delta)$-bimodule generated by
the arrows from $a$ to the quiver of $\Lambda$. It is easy to see that $\Gamma$ is a split extension of $\Delta$
by the nilpotent bimodule $_{\Delta}E_{\Delta}$, and $E_\Delta \cong M_\Delta$ while $D(_\Delta E)\cong S^{t}$
where $S$ is the simple module corresponding to the point $a$ and $t=|M|$ (\cite{AZ1}).

In the rest of this subsection, $\Gamma$ is a one-point extension of $\Lambda$ by a module $M$ in $\mod \Lambda$,
and $e_{a}$ is the idempotent corresponding to the extension point $a$ and $\Delta:=\Lambda \times k$.

\begin{remark}\label{3.12}
\begin{enumerate}
\item[]
\item[(1)] The algebra $\Gamma$ is a $\Delta$-$\Delta$-bimodule and a $\Lambda$-$\Lambda$-bimodule.
\item[(2)] The algebra $\Delta$ is a $\Lambda$-$\Lambda$-bimodule.
\item[(3)] For any $\Lambda$-module $X$, it can be seen as a $\Delta$-module or a $\Gamma$-module.
In fact, $$X_\Gamma\cong X_\Delta\otimes_\Delta \Gamma\cong X_\Lambda \otimes_\Lambda\Gamma.$$
\item[(4)] For any $\Delta$-module $N$, we have $N_\Delta\cong Y_\Delta\oplus S^t$ for some $t\geq 0$,
where $Y$ is a $\Lambda$-module.
\end{enumerate}
\end{remark}

We need the following two easy observations.

\begin{lemma}\label{3.13}
For any $X\in\mod\Lambda$, we have $X\otimes_\Delta E=0$.
\end{lemma}

\begin{proof}
Considering the projective presentation
$$e_2\Lambda {\longrightarrow}e_1\Lambda{\longrightarrow} X{\longrightarrow} 0$$
of $X$ in $\mod \Lambda$ with $e_1,e_2$ idempotents of $\Lambda$, we get the projective presentation
$$e_2\Delta {\longrightarrow}e_1\Delta{\longrightarrow} X{\longrightarrow} 0$$
of $X$ in $\mod \Delta$. Applying the functor $-\otimes_\Delta E$ yields the following exact sequence
$$e_2E {\longrightarrow}e_1E{\longrightarrow} X\otimes_\Delta E{\longrightarrow} 0.$$
Since $E$ is generated by the arrows from $a$ to the quiver of $\Lambda$, we have $e_1E=0=e_2E=0$.
Hence $X\otimes_\Delta E=0$.
\end{proof}

\begin{lemma}\label{3.14}
$S\otimes_\Delta E\cong M_\Delta$.
\end{lemma}

\begin{proof}
It follows Lemma \ref{3.13} and the following isomorphism
$$M_\Delta\cong E_\Delta\cong \Delta\otimes_\Delta E\cong(S\oplus \Lambda)\otimes_\Delta E
\cong(S\otimes_\Delta E)\oplus(\Lambda\otimes_\Delta E).$$
\end{proof}

Note that basic support $\tau$-tilting modules in $\mod \Lambda$ are exactly those forms
$T$ and $T\oplus S$ where $T$ is a support $\tau$-tilting $\Lambda$-module. Hence support $\tau$-tilting pairs in $\mod \Delta$ are exactly
those forms $(T,P\oplus S)$ and $(T\oplus S,P)$ where $P$ is a projective $\Lambda$-module such that
$(T,P)$ is a support $\tau$-tilting pair in $\mod \Lambda$.
As a consequence of Theorem \ref{3.1}, we also have the following

\begin{proposition}\label{3.15}
Let $\Gamma$ be a one-point extension of $\Lambda$ by a module $M$ in $\mod \Lambda$, and let $e_{a}$
be the idempotent corresponding to the extension point $a$. Then for a pair $(T,P)$ in $\mod \Lambda$ with $P$ projective, we have
\begin{enumerate}
\item[(1)] $(T_\Gamma, P_\Gamma \oplus e_{a}\Gamma)$ is a support $\tau$-tilting pair in $\mod \Gamma$
if and only if $(T_\Lambda ,P_\Lambda)$ is a support $\tau$-tilting pair in $\mod \Lambda$.
\item[(2)] $(T_\Gamma \oplus e_{a}\Gamma, P_\Gamma)$ is a support $\tau$-tilting pair in $\mod \Gamma$
if and only if $(T_\Lambda ,P_\Lambda)$ is a support $\tau$-tilting pair in $\mod \Lambda$ and
$\Hom_\Lambda(M_\Lambda,\tau T_\Lambda)=0=\Hom_\Lambda(P_\Lambda,M_\Lambda)$.
\end{enumerate}
\end{proposition}

\begin{proof}
It follows Theorem \ref{3.1} and Lemmas \ref{3.13} and \ref{3.14}.
\end{proof}

Putting $P=0$ in Proposition \ref{3.15}, we get the following

\begin{corollary}\label{3.16}
\begin{enumerate}
\item[]
\item[(1)] $T_\Gamma$ is $\tau$-tilting if and only if $T_\Lambda$ is $\tau$-tilting.
\item[(2)] $T_\Gamma \oplus e_{a}\Gamma$ is $\tau$-tilting
in $\mod \Gamma$ if and only if $T$ is $\tau$-tilting in $\mod \Lambda$ and $\Hom_\Lambda(M_\Lambda,\tau T_\Lambda)=0$.
\end{enumerate}
\end{corollary}

If $\Gamma$ is a one-point extension of $\Lambda$ by a non-zero module $M_\Lambda$, then there exists an
idempotent $e\in\Lambda$ such that $\Hom_\Lambda(e\Lambda, M_\Lambda)\neq 0$. Note that there are $\tau$-tilting
$\Lambda/<e>$-modules. So, by Proposition \ref{3.15}(2), we have the following

\begin{corollary}\label{3.17}
Let $\Gamma$ be a one-point extension of $\Lambda$ by a non-zero module $M_\Lambda$.
Then there exists a support $\tau$-tilting $\Lambda$-module such that $T_\Gamma \oplus e_{a}\Gamma$ is not  support $\tau$-tilting.
\end{corollary}

\section{Examples}

In this section, we give two examples to illustrate the results obtained in Section 3.
All indecomposable modules are denoted by their Loewy series.

\begin{example}\label{4.1}
{\rm Let $\Sigma$ be a finite dimensional $k$-algebra given by the quiver
$$1 {\longrightarrow} 2 {\longrightarrow} 3.$$ Then ${T=\smallmatrix 1
\endsmallmatrix}{\smallmatrix 1\\2\\3
\endsmallmatrix}{\smallmatrix 3
\endsmallmatrix}$ is a tilting $\Sigma$-module. The endomorphism algebra $\Lambda$ of $T$ is a tilted algebra given by the quiver
$$1 \stackrel{\alpha}{\longrightarrow} 2 \stackrel{\beta}{\longrightarrow} 3$$
with the relation $\alpha\beta=0$. The cluster-titlted algebra $\Gamma$ corresponding to $\Lambda$ is given by
the following quiver
$$\xymatrix{
 & 2  \ar^{\beta}[rd]  &   \\
1 \ar[ru]^{\alpha} &     &  3\ar^{\gamma}[ll]
}$$
with relations $\alpha\beta=0$, $\beta\gamma=0$ and $\gamma\alpha=0$, and $\Gamma$ is a split-by-nilpotent extension of $\Lambda$.

Note that ${\smallmatrix 3\endsmallmatrix}$ is the unique indecomposable module in $\mod \Lambda$ with injective dimension two.
So for any indecomposable module $W$ not isomorphic to ${\smallmatrix 3\endsmallmatrix}$, we have $\tau^{-1}\Omega^{-1}W=0$.
Because $$\xymatrix{
{\smallmatrix 0
\endsmallmatrix}\ar[r] &{\smallmatrix 3
\endsmallmatrix}\ar[r] & {\smallmatrix 2\\3
\endsmallmatrix}\ar[r] & {\smallmatrix 1\\2
\endsmallmatrix}\ar[r] & {\smallmatrix 1
\endsmallmatrix}\ar[r] &{\smallmatrix 0
\endsmallmatrix}
}$$ is a minimal injective resolution of $3$,
we have $\tau^{-1}\Omega^{-1}{\smallmatrix 3
\endsmallmatrix} =\tau^{-1}{\smallmatrix 2
\endsmallmatrix} ={\smallmatrix 1
\endsmallmatrix}$.

Let $(T_i,P_i)$ be a support $\tau$-tilting pair in $\mod \Lambda$ and $\widetilde{T}_i:=T_i\otimes_\Lambda \Gamma$ for each $i$.
We list $\widetilde{T}_i$, $\tau^{-1}\Omega^{-1}T_i$ and $\Hom_\Lambda(P_i,\tau^{-1}\Omega^{-1}T_i)$ in the following table.

\vspace{0.2cm}

\begin{tabularx}{\textwidth}{|X|X|X|c|c|}
$T_{i}$& $P_{i}$ & $\widetilde{T}_i=T_i\otimes \Gamma$ &$\tau^{-1}\Omega^{-1}T_i$ &$\Hom_\Lambda(P_i,\tau^{-1}\Omega^{-1}T_i)$ \\ \hline
${T_1=\smallmatrix 1\\2
\endsmallmatrix}{\smallmatrix 2\\3
\endsmallmatrix}{\smallmatrix 3
\endsmallmatrix}$&${\smallmatrix 0
\endsmallmatrix}$&${\widetilde{T}_1=\smallmatrix 1\\2
\endsmallmatrix}{\smallmatrix 2\\3
\endsmallmatrix}{\smallmatrix 3\\1
\endsmallmatrix}$&${\smallmatrix 1
\endsmallmatrix}$&${\smallmatrix 0
\endsmallmatrix}$\\
${T_2=\smallmatrix 1\\2
\endsmallmatrix}{\smallmatrix 2\\3
\endsmallmatrix}{\smallmatrix 2
\endsmallmatrix}$&${\smallmatrix 0
\endsmallmatrix}$&${\widetilde{T}_2=\smallmatrix 1\\2
\endsmallmatrix}{\smallmatrix 2\\3
\endsmallmatrix}{\smallmatrix 2
\endsmallmatrix}$&${\smallmatrix 0
\endsmallmatrix}$&${\smallmatrix 0
\endsmallmatrix}$\\
${T_3=\smallmatrix 1\\2
\endsmallmatrix}{\smallmatrix 1
\endsmallmatrix}{\smallmatrix 3
\endsmallmatrix}$&${\smallmatrix 0
\endsmallmatrix}$&${\widetilde{T}_3=\smallmatrix 1\\2
\endsmallmatrix}{\smallmatrix 1
\endsmallmatrix}{\smallmatrix 3\\1
\endsmallmatrix}$&${\smallmatrix 1
\endsmallmatrix}$&${\smallmatrix 0
\endsmallmatrix}$\\ \hline
${T_4=\smallmatrix 2\\3
\endsmallmatrix}{\smallmatrix 3
\endsmallmatrix}$&${\smallmatrix 1\\2
\endsmallmatrix}$&${\widetilde{T}_4=\smallmatrix 2\\3
\endsmallmatrix}{\smallmatrix 3\\1
\endsmallmatrix}$&${\smallmatrix 1
\endsmallmatrix}$&${\smallmatrix \neq 0
\endsmallmatrix}$\\ \hline
${T_5=\smallmatrix 1\\2
\endsmallmatrix}{\smallmatrix 2
\endsmallmatrix}$&${\smallmatrix 3
\endsmallmatrix}$&${\widetilde{T}_5=\smallmatrix 1\\2
\endsmallmatrix}{\smallmatrix 2
\endsmallmatrix}$&${\smallmatrix 0
\endsmallmatrix}$&${\smallmatrix 0
\endsmallmatrix}$\\
${T_6=\smallmatrix 1
\endsmallmatrix}{\smallmatrix 3
\endsmallmatrix}$&${\smallmatrix 2\\3
\endsmallmatrix}$&${\widetilde{T}_6=\smallmatrix 1
\endsmallmatrix}{\smallmatrix 3\\1
\endsmallmatrix}$&${\smallmatrix 1
\endsmallmatrix}$&${\smallmatrix 0
\endsmallmatrix}$\\
${T_7=\smallmatrix 1\\2
\endsmallmatrix}{\smallmatrix 1
\endsmallmatrix}$&${\smallmatrix 3
\endsmallmatrix}$&${\widetilde{T}_7=\smallmatrix 1\\2
\endsmallmatrix}{\smallmatrix 1
\endsmallmatrix}$&${\smallmatrix 0
\endsmallmatrix}$&${\smallmatrix 0
\endsmallmatrix}$\\
${T_8=\smallmatrix 2\\3
\endsmallmatrix}{\smallmatrix 2
\endsmallmatrix}$&${\smallmatrix 1\\2
\endsmallmatrix}$&${\widetilde{T}_8=\smallmatrix 2\\3
\endsmallmatrix}{\smallmatrix 2
\endsmallmatrix}$&${\smallmatrix 0
\endsmallmatrix}$&${\smallmatrix 0
\endsmallmatrix}$\\
${T_9=\smallmatrix 1
\endsmallmatrix}$&${\smallmatrix 2\\3
\endsmallmatrix}{\smallmatrix  3
\endsmallmatrix}$&${\widetilde{T}_9=\smallmatrix 1
\endsmallmatrix}$&${\smallmatrix 0
\endsmallmatrix}$&${\smallmatrix 0
\endsmallmatrix}$\\
${T_{10}=\smallmatrix 2
\endsmallmatrix}$&${\smallmatrix 1\\2
\endsmallmatrix}{\smallmatrix  3
\endsmallmatrix}$&${\widetilde{T}_{10}=\smallmatrix 2
\endsmallmatrix}$&${\smallmatrix 0
\endsmallmatrix}$&${\smallmatrix 0
\endsmallmatrix}$\\ \hline
${T_{11}=\smallmatrix 3
\endsmallmatrix}$&${\smallmatrix 1\\2
\endsmallmatrix}{\smallmatrix  2\\3
\endsmallmatrix}$&${\widetilde{T}_{11}=\smallmatrix 3\\1
\endsmallmatrix}$&${\smallmatrix 1
\endsmallmatrix}$&${\smallmatrix \neq0
\endsmallmatrix}$\\ \hline
${T_{12}=\smallmatrix 0
\endsmallmatrix}$&${\smallmatrix 1\\2
\endsmallmatrix}{\smallmatrix  2\\3
\endsmallmatrix}{\smallmatrix  3
\endsmallmatrix}$&${\widetilde{T}_{12}=\smallmatrix 0
\endsmallmatrix}$&${\smallmatrix 0
\endsmallmatrix}$&${\smallmatrix 0
\endsmallmatrix}$\\
\end{tabularx}

\vspace{0.2cm}

A simple calculation yields
$$\Hom_\Lambda(\tau^{-1}\Omega^{-1}T_1, \tau T_1)=0,$$
$$\Hom_\Lambda(\tau^{-1}\Omega^{-1}T_3, \tau T_3)\cong\Hom_\Lambda(\tau^{-1}\Omega^{-1}T_3, \tau {\smallmatrix 1
\endsmallmatrix})\cong\Hom_\Lambda({\smallmatrix 1\endsmallmatrix}, {\smallmatrix2\endsmallmatrix})=0,$$
$$\Hom_\Lambda(\tau^{-1}\Omega^{-1}T_6, \tau T_6)\cong\Hom_\Lambda(\tau^{-1}\Omega^{-1}T_6, \tau {\smallmatrix 1
\endsmallmatrix})\cong\Hom_\Lambda({\smallmatrix 1\endsmallmatrix}, {\smallmatrix 2
\endsmallmatrix})=0.$$
Thus all $\widetilde{T}_1$, $\widetilde{T}_2$, $\widetilde{T}_3$, $\widetilde{T}_5$, $\widetilde{T}_6$,
$\widetilde{T}_7$, $\widetilde{T}_8$, $\widetilde{T}_9$, $\widetilde{T}_{10}$ and $\widetilde{T}_{12}$
are support $\tau$-tilting, and neither $\widetilde{T}_4$ nor $\widetilde{T}_{11}$ is support $\tau$-tilting by Proposition \ref{3.4}.
We draw the Hasse quivers $Q(s\tau$-$\tilt \Lambda)$ and $Q(s\tau$-$\tilt \Gamma)$ as follows, where $M_{(T_i)}$ stands for $(T_i=M)$.
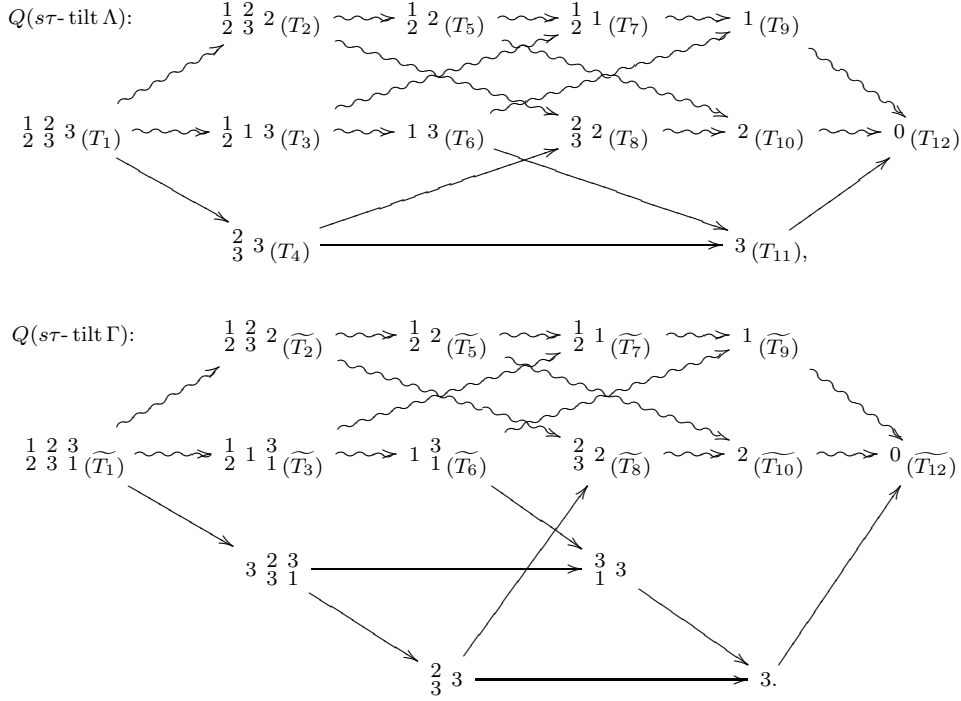
\[\xymatrix{ {\smallmatrix Q(s\tau{\text -}\tilt \Lambda):
\endsmallmatrix} &{\smallmatrix 1\\2
\endsmallmatrix}{\smallmatrix 2\\3
\endsmallmatrix}{\smallmatrix 2
\endsmallmatrix}_{(T_2)}\ar@{~>}[r]\ar@{~>}[rrd]&{\smallmatrix 1\\2
\endsmallmatrix}{\smallmatrix 2
\endsmallmatrix}_{(T_5)}\ar@{~>}[r]\ar@{~>} [rrd]&{\smallmatrix 1\\2
\endsmallmatrix}{\smallmatrix 1
\endsmallmatrix}_{(T_7)}\ar@{~>}[r]&{\smallmatrix 1
\endsmallmatrix}_{(T_9)}\ar@{~>}[rd]&\\
{\smallmatrix 1\\2
\endsmallmatrix}{\smallmatrix 2\\3
\endsmallmatrix}{\smallmatrix 3
\endsmallmatrix}_{(T_1)}\ar@{~>}[r]\ar@{~>}[ur]\ar[rd]&{\smallmatrix 1\\2
\endsmallmatrix}{\smallmatrix 1
\endsmallmatrix}{\smallmatrix 3
\endsmallmatrix}_{(T_3)}\ar@{~>}[r]\ar@{~>}[rru]&{\smallmatrix 1
\endsmallmatrix}{\smallmatrix 3
\endsmallmatrix}_{(T_6)}\ar@{~>}[rru]\ar[rrd]&{\smallmatrix 2\\3
\endsmallmatrix}{\smallmatrix 2
\endsmallmatrix}_{(T_8)}\ar@{~>}[r]&{\smallmatrix 2
\endsmallmatrix}_{(T_{10})}\ar@{~>}[r]&{\smallmatrix 0
\endsmallmatrix}_{(T_{12})}\\
&{\smallmatrix 2\\3
\endsmallmatrix}{\smallmatrix 3
\endsmallmatrix}_{(T_4)}\ar[rrr]\ar[rru]&&&{\smallmatrix 3
\endsmallmatrix}_{(T_{11}),}\ar[ru]&
}\]

\[\xymatrix{ {\smallmatrix Q(s\tau{\text -}\tilt \Gamma):
\endsmallmatrix} &{\smallmatrix 1\\2
\endsmallmatrix}{\smallmatrix 2\\3
\endsmallmatrix}{\smallmatrix 2
\endsmallmatrix}_{(\widetilde{T_2})}\ar@{~>}[r]\ar@{~>}[rrd]&{\smallmatrix 1\\2
\endsmallmatrix}{\smallmatrix 2
\endsmallmatrix}_{(\widetilde{T_5})}\ar@{~>}[r]\ar@{~>} [rrd]&{\smallmatrix 1\\2
\endsmallmatrix}{\smallmatrix 1
\endsmallmatrix}_{(\widetilde{T_7})}\ar@{~>}[r]&{\smallmatrix 1
\endsmallmatrix}_{(\widetilde{T_9})}\ar@{~>}[rd]&\\
{\smallmatrix 1\\2
\endsmallmatrix}{\smallmatrix 2\\3
\endsmallmatrix}{\smallmatrix 3\\1
\endsmallmatrix}_{(\widetilde{T_1})}\ar@{~>}[r]\ar@{~>}[ur]\ar[rd]&{\smallmatrix 1\\2
\endsmallmatrix}{\smallmatrix 1
\endsmallmatrix}{\smallmatrix 3\\1
\endsmallmatrix}_{(\widetilde{T_3})}\ar@{~>}[r]\ar@{~>}[rru]&{\smallmatrix 1
\endsmallmatrix}{\smallmatrix 3\\1
\endsmallmatrix}_{(\widetilde{T_6})}\ar@{~>}[rru]\ar[rd]&{\smallmatrix 2\\3
\endsmallmatrix}{\smallmatrix 2
\endsmallmatrix}_{(\widetilde{T_8})}\ar@{~>}[r]&{\smallmatrix 2
\endsmallmatrix}_{(\widetilde{T_{10}})}\ar@{~>}[r]&{\smallmatrix 0
\endsmallmatrix}_{(\widetilde{T_{12}})}\\
&{\smallmatrix 3
\endsmallmatrix}{\smallmatrix 2\\3
\endsmallmatrix}{\smallmatrix 3\\1
\endsmallmatrix}\ar[rr]\ar[rd]&&{\smallmatrix 3\\1
\endsmallmatrix}{\smallmatrix 3
\endsmallmatrix}\ar[rd]&&\\
&&{\smallmatrix 2\\3
\endsmallmatrix}{\smallmatrix 3
\endsmallmatrix}\ar[rr]\ar[ruu]&&{\smallmatrix 3.
\endsmallmatrix}\ar[ruu]&
}\]
We draw those arrows in  $fQ(s\tau$-$\tilt \Gamma)$ and $fQ(s\tau$-$\tilt \Lambda)$
by\xymatrix{\ar@{~>}[r] &}. Their underlying graphs and corresponding arrows are identical.}
\end{example}

\begin{example}\label{4.2}
{\rm Let $\Lambda$ be a finite dimensional $k$-algebra given by the quiver
$$2 {\longrightarrow} 3.$$ Considering the one-point extension of $\Lambda$ by the simple module corresponding to the point $2$,
the algebra $\Gamma=\Lambda[2]$ is given by the quiver
$$1 \stackrel{\alpha}{\longrightarrow} 2 \stackrel{\beta}{\longrightarrow} 3$$ with the relation  $\alpha\beta=0$.
Let $\Delta=\Lambda \times k$. The following is the Hasse quiver of $\Lambda$.
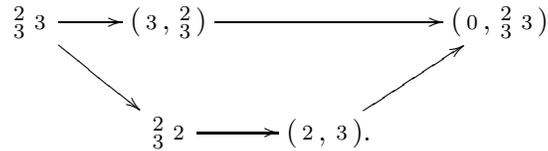
\[\xymatrix{{\smallmatrix 2\\3
\endsmallmatrix}{\smallmatrix 3
\endsmallmatrix}\ar[r]\ar[rd] &({\smallmatrix 3
\endsmallmatrix},{\smallmatrix 2\\3
\endsmallmatrix})\ar[rr]&&({\smallmatrix 0
\endsmallmatrix},{\smallmatrix 2\\3
\endsmallmatrix}{\smallmatrix 3
\endsmallmatrix})\\
 &{\smallmatrix 2\\3
\endsmallmatrix}{\smallmatrix 2
\endsmallmatrix}\ar[r]&({\smallmatrix 2
\endsmallmatrix},{\smallmatrix 3
\endsmallmatrix}).\ar[ru]&
}\]
By Proposition \ref{3.15}(1), all ${\smallmatrix 2\\3
\endsmallmatrix}{\smallmatrix 3
\endsmallmatrix},{\smallmatrix 3
\endsmallmatrix}, {\smallmatrix 0
\endsmallmatrix}, {\smallmatrix 2\\3
\endsmallmatrix}{\smallmatrix 2
\endsmallmatrix}$ and ${\smallmatrix 2
\endsmallmatrix}$ are support $\tau$-tilting $\Gamma$-modules. From support $\tau$-tilting $\Lambda$-modules ${\smallmatrix 3
\endsmallmatrix}$ and ${\smallmatrix 0\endsmallmatrix}$, it is easy to get two support $\tau$-tilting $\Delta$-pairs $({\smallmatrix 3
\endsmallmatrix}{\smallmatrix 1
\endsmallmatrix}, {\smallmatrix 2\\3
\endsmallmatrix})$ and $({\smallmatrix 1
\endsmallmatrix}, {\smallmatrix 2\\3
\endsmallmatrix}{\smallmatrix 3
\endsmallmatrix})$.
Since $\Hom_\Lambda({\smallmatrix 2\\3
\endsmallmatrix}, {\smallmatrix2
\endsmallmatrix})\neq 0$, it follows from Proposition \ref{3.15}(2) that neither
${\smallmatrix 3
\endsmallmatrix}{\smallmatrix 1\\2
\endsmallmatrix}$ nor ${\smallmatrix 1\\2
\endsmallmatrix}$ is a support $\tau$-tilting $\Gamma$-module.
A simple calculation yields that all ${\smallmatrix 2\\3
\endsmallmatrix}{\smallmatrix 3
\endsmallmatrix}{\smallmatrix 1\\2
\endsmallmatrix}$, ${\smallmatrix 2\\3
\endsmallmatrix}{\smallmatrix 2
\endsmallmatrix}{\smallmatrix 1\\2
\endsmallmatrix}$ and ${\smallmatrix 2
\endsmallmatrix}{\smallmatrix 1\\2
\endsmallmatrix}$ are support $\tau$-tilting $\Gamma$-modules also by Proposition \ref{3.15}(2).

Now we draw $Q(s\tau$-$\tilt \Delta)$ and $Q(s\tau$-$\tilt \Gamma)$ as follows.
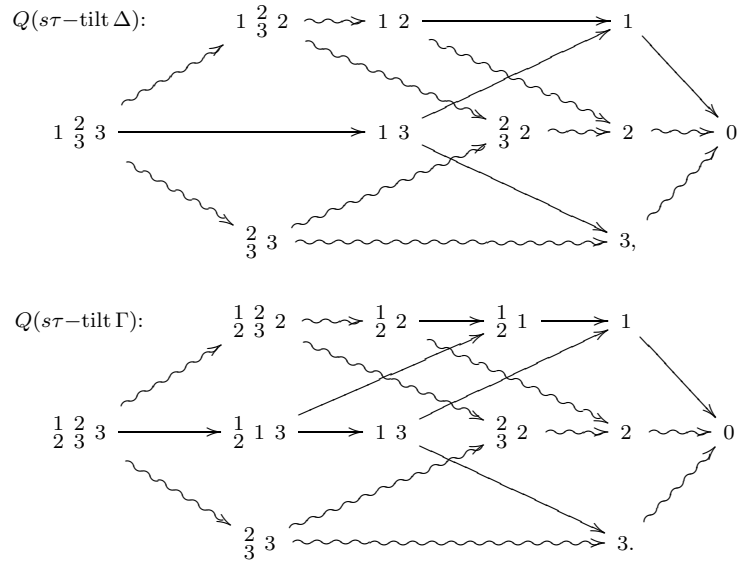
\[\xymatrix{ {\smallmatrix Q(s\tau-\tilt \Delta):
\endsmallmatrix} &{\smallmatrix 1
\endsmallmatrix}{\smallmatrix 2\\3
\endsmallmatrix}{\smallmatrix 2
\endsmallmatrix}\ar@{~>}[r]\ar@{~>}[rrd]&{\smallmatrix 1
\endsmallmatrix}{\smallmatrix 2
\endsmallmatrix}\ar[rr]\ar@{~>} [rrd]&&{\smallmatrix 1
\endsmallmatrix}\ar[rd]&\\
{\smallmatrix 1
\endsmallmatrix}{\smallmatrix 2\\3
\endsmallmatrix}{\smallmatrix 3
\endsmallmatrix}\ar[rr]\ar@{~>}[ur]\ar@{~>}[rd]&&{\smallmatrix 1
\endsmallmatrix}{\smallmatrix 3
\endsmallmatrix}\ar[rru]\ar[rrd]&{\smallmatrix 2\\3
\endsmallmatrix}{\smallmatrix 2
\endsmallmatrix}\ar@{~>}[r]&{\smallmatrix 2
\endsmallmatrix}\ar@{~>}[r]&{\smallmatrix 0
\endsmallmatrix}\\
&{\smallmatrix 2\\3
\endsmallmatrix}{\smallmatrix 3
\endsmallmatrix}\ar@{~>}[rrr]\ar@{~>}[rru]&&&{\smallmatrix 3,
\endsmallmatrix}\ar@{~>}[ru]&
}\]
\[\xymatrix{ {\smallmatrix Q(s\tau-\tilt \Gamma):
\endsmallmatrix} &{\smallmatrix 1\\2
\endsmallmatrix}{\smallmatrix 2\\3
\endsmallmatrix}{\smallmatrix 2
\endsmallmatrix}\ar@{~>}[r]\ar@{~>}[rrd]&{\smallmatrix 1\\2
\endsmallmatrix}{\smallmatrix 2
\endsmallmatrix}\ar[r]\ar@{~>} [rrd]&{\smallmatrix 1\\2
\endsmallmatrix}{\smallmatrix 1
\endsmallmatrix}\ar[r]&{\smallmatrix 1
\endsmallmatrix}\ar[rd]&\\
{\smallmatrix 1\\2
\endsmallmatrix}{\smallmatrix 2\\3
\endsmallmatrix}{\smallmatrix 3
\endsmallmatrix}\ar[r]\ar@{~>}[ur]\ar@{~>}[rd]&{\smallmatrix 1\\2
\endsmallmatrix}{\smallmatrix 1
\endsmallmatrix}{\smallmatrix 3
\endsmallmatrix}\ar[r]\ar[rru]&{\smallmatrix 1
\endsmallmatrix}{\smallmatrix 3
\endsmallmatrix}\ar[rru]\ar[rrd]&{\smallmatrix 2\\3
\endsmallmatrix}{\smallmatrix 2
\endsmallmatrix}\ar@{~>}[r]&{\smallmatrix 2
\endsmallmatrix}\ar@{~>}[r]&{\smallmatrix 0
\endsmallmatrix}\\
&{\smallmatrix 2\\3
\endsmallmatrix}{\smallmatrix 3
\endsmallmatrix}\ar@{~>}[rrr]\ar@{~>}[rru]&&&{\smallmatrix 3.
\endsmallmatrix}\ar@{~>}[ru]&
}\]
We also draw those arrows in  $fQ(s\tau$-$\tilt \Delta)$ and $fQ(s\tau$-$\tilt \Gamma)$
by\xymatrix{\ar@{~>}[r] &}. Their underlying graphs and corresponding arrows are identical.}
\end{example}

\vspace{0.5cm}

{\bf Acknowledgements.} The authors thank the referee for the useful and detailed suggestions.

{\small
}
\end{document}